\documentclass{amsart}

\usepackage{amsmath,amsfonts}
\usepackage[english]{babel}

\newtheorem{theorem}{Theorem}[section]
\newtheorem{lemma}[theorem]{Lemma}
\newtheorem{definition}[theorem]{Definition}
\newtheorem{proposition}[theorem]{Proposition}
\newtheorem{corollary}[theorem]{Corollary}

\newtheorem{example}[theorem]{Example}
\newtheorem{problem}[theorem]{Problem}
\numberwithin{equation}{section}

\begin{document}

\title{Classification of separately continuous functions with values in $\sigma$-metrizable spaces}

\author{Olena Karlova}

\begin{abstract}
We prove that every vertically nearly separately continuous function defined on a product of a strong PP-space and a topological space and with values in a strongly $\sigma$-metrizable space with a special stratification,    is a pointwise limit of continuous functions.
\end{abstract}



\maketitle

\section{Introduction}

Let $X$, $Y$ and $Z$ be topological spaces.

By $C(X,Y)$ we denote the collection of all continuous mappings between $X$ and $Y$.

For a mapping $f:X\times Y\to Z$ {and a} point $(x,y)\in X\times Y$ we write
$$f^x(y)=f_y(x)=f(x,y).$$

We say that a mapping $f:X\times Y\to Z$ {is} {\it separately continuous}, $f\in CC(X\times Y,Z)$, if $f^x\in C(Y,Z)$ and $f_y\in C(X,Z)$ for every point $(x,y)\in X\times Y$. A mapping $f:X\times Y\to Z$ is said to be {\it vertically nearly separately continuous}, $f\in C\overline{C}(X\times Y,Z)$, if $f_y\in C(X,Z)$ for every $y\in Y$ and there exists a dense set $D\subseteq X$ such that $f^x\in C(Y,Z)$ for all $x\in D$.

Let $B_0(X,Y)=C(X,Y)$. Assume that the classes $B_\xi(X,Y)$ are already defined for all $\xi<\alpha$, where $\alpha<\omega_1$. Then $f:X\to Y$ is {said to be} {\it of the $\alpha$-th Baire class}, $f\in B_\alpha(X,Y)$, if $f$ is a pointwise limit of a sequence of mappings $f_n\in
B_{\xi_n}(X,Y)$, where $\xi_n<\alpha$. In particular, $f\in B_1(X,Y)$ if it is a pointwise limit of a sequence of continuous mappings.

In 1898 H.~Lebesgue \cite{Leb1} proved that every real-valued separately continuous functions of two real variables is of the first Baire class.
Lebesgue's theorem was generalized by many mathematicians (see \cite{Hahn, Moran, Ru, Vera, SobPP, BanakhT, Bu1, KMas3, K4, M} and the references given there). W.~Rudin\cite{Ru} showed that $C\overline{C}(X\times Y,Z)\subseteq B_{1}(X\times Y,Z)$ if $X$ is a metrizable space, $Y$ a topological space and $Z$ a locally convex space. Naturally the following question {has been} arose, which is still unanswered.

\begin{problem}\label{P1}
  Let $X$ be a metrizable space, $Y$ a topological space and $Z$ a topological vector space. Does every separately continuous mapping $f:X\times Y\to Z$ belong to the first Baire class?
\end{problem}

V.~Maslyuchenko and A.~Kalancha \cite{KMas3} showed that the answer is positive, when  $X$ is a metrizable space with finite \v{C}ech-Lebesgue dimension.
T.~Banakh \cite{BanakhT} gave {a} positive answer in the case when $X$ is a  metrically quarter-stratifiable  paracompact strongly countably dimensional space and $Z$ is an equiconnected space.  In \cite{Ka} it was shown that the answer to Problem~\ref{P1} is positive for a metrizable spaces $X$ and $Y$ and a metrizable arcwise connected and locally arcwise connected space $Z$. It was pointed out in \cite{KaMas} that $CC(X\times Y,Z)\subseteq B_1(X\times Y,Z)$ if $X$ is a metrizable space, $Y$ is a topological space and $Z$ is an equiconnected strongly $\sigma$-metrizable space with a stratification $(Z_n)_{n=1}^\infty$ (see the definitions below), where $Z_n$ is a metrizable arcwise connected and locally arcwise connected space for every $n\in\mathbb N$.

In the given paper we generalize the above-mentioned result from \cite{KaMas} {to} the case of vertically nearly separately continuous mappings. To do this, we introduce the class of strong PP-spaces which includes the class of all metrizable spaces. In Section~\ref{sec:PP} we investigate some properties of strong PP-spaces. In Section~\ref{sec:Leb} we establish {an} auxiliary result which {generalizes} the famous Kuratowski-Montgomery theorem (see \cite{Ku2} and \cite{Mont}). Finally, in Section~\ref{sec:Baire} we prove that the inclusion $C\overline{C}(X\times Y,Z)\subseteq B_1(X\times Y,Z)$ holds if $X$ is a strong PP-space, Y is a topological space and $Z$ is a contractible space with a stratification $(Z_n)_{n=1}^\infty$, where $Z_n$ is a metrizable arcwise connected and locally arcwise connected space for every $n\in\mathbb N$.

\section{Preliminary observations}

A subset $A$ of a topological space $X$ is {\it a zero (co-zero) set} if $A=f^{-1}(0)$ ($A=f^{-1}((0,1])$) for some continuous function $f:X\to [0,1]$.

Let ${\mathcal G}_0^*$ and ${\mathcal
F}_0^*$ be  collections of all co-zero and zero subsets of $X$, respectively. Assume that the classes ${\mathcal G}_\xi^*$ and ${\mathcal F}_\xi^*$ are defined for all $\xi<\alpha$, where
$0<\alpha<\omega_1$. Then, if $\alpha$ is odd, the class ${\mathcal G}_\alpha^*$ (${\mathcal F}_\alpha^*$) is consists of all countable intersections  (unions) of sets of lower classes, and, if $\alpha$ is even, the class ${\mathcal G}_\alpha^*$ (${\mathcal F}_\alpha^*$) is consists of all countable unions  (intersections) of sets of lower classes. The classes  ${\mathcal F}_\alpha^*$ for odd $\alpha$ and  ${\mathcal G}_\alpha^*$ for even  $\alpha$ are said to be
{\it functionally additive}, and the classes ${\mathcal F}_\alpha^*$
for even $\alpha$ and ${\mathcal G}_\alpha^*$ for odd
$\alpha$ are called {\it functionally multiplicative}. If a set belongs to the $\alpha$'th functionally additive and functionally multiplicative class, then it is called {\it
functionally ambiguous of the $\alpha$'th class}.
Remark that $A\in {\mathcal F}_\alpha^*$ if and only {if $X\setminus A\in {\mathcal G}_\alpha^*$}.

If a set $A$ is of the first functionally additive /multiplicative/ class, we say that $A$ is an $F_\sigma^*$ /$G_\delta^*$/ set.

Let us observe that  if $X$ is a perfectly normal space (i.e. a normal space in which every closed subset {is} of the type $G_\delta$), then functionally additive and functionally multiplicative classes coincide with ordinary additive and multiplicative classes respectively, since every open set in $X$ is functionally open.

\begin{lemma}\label{l:3}
  Let $\alpha\ge 0$, $X$ be a topological space and {let} $A\subseteq X$ {be} of the $\alpha$'th functionally multiplicative class. Then there exists {a function  $f\in   B_{\alpha}(X,[0,1])$ such} that $A=f^{-1}(0)$.
\end{lemma}

\begin{proof}
The hypothesis of the lemma is obvious if $\alpha=0$.

Suppose the assertion of the lemma is true for all $\xi<\alpha$ and let $A$ {be} a set of the $\alpha$'th functionally multiplicative class. Then
   $A=\bigcap\limits_{n=1}^\infty A_n$, where $A_n$ belong to the $\alpha_n$'th functionally additive  class {with} $\alpha_n<\alpha$ for all $n\in\mathbb N$. By the assumption, there exists a sequence of functions \mbox{$f_n\in B_{\alpha_n}(X,[0,1])$} such that  $A_n=f_n^{-1}((0,1])$.
   Notice that for every $n$ the characteristic function $\chi_{A_n}$ of $A_n$ belongs to the $\alpha$-th Baire class. Indeed, setting
   $h_{n,m}(x)=\sqrt[m]{f_n(x)}$, we obtain {a} sequence of functions $h_{n,m}\in B_{\alpha_n}(X,[0,1])$ which is pointwise convergent to  $\chi_{A_n}$. Now let
  $$
  f(x)=1-\sum\limits_{n=1}^\infty\frac{1}{2^n}\chi_{A_n}(x).
  $$
  for all $x\in X$.
  Then $f\in B_\alpha(X,[0,1])$ as a sum of a uniform convergent series of functions of the {$\alpha$'th class}. Moreover, it is easy to see that  $A=f^{-1}(0)$.\hfill$\Box$
\end{proof}

A topological space $X$ is
  \begin{itemize}
    \item   {\it
equiconnected} if there exists a continuous mapping $\lambda:X\times X\times [0,1]\to X$ such that

 (1) $\lambda(x,y,0)=x$;

 (2) $\lambda(x,y,1)=y$;

 (3) $\lambda(x,x,t)=x$\\ for all $x,y\in X$ and $t\in [0,1]$.

\item {\it contractible} if there exist  $x^*\in X$ and a continuous function $\gamma:X\times
[0,1]\to X$ such that $\gamma(x,0)=x$ and $\gamma(x,1)=x^*$. A contractible space $X$ with such a point $x^*$ and such a function $\gamma$ we denote by $(X,x^*,\gamma)$.

\end{itemize}

Remark that every convex subset $X$ of a topological vector space is equiconnected, where $\lambda:X\times X\times [0,1]\to X$ is defined by the formula $\lambda(x,y,t)=(1-t)x+ty$, $x,y\in X$, $t\in [0,1]$.

It is easily seen that a topological space $X$ is contractible if and only if there exists {a continuous function $\lambda:X\times X\times [0,1]\to X$ such} that $\lambda(x,y,0)=x$ and  $\lambda(x,y,1)=y$ for all $x,y\in X$. Indeed, if $(X,x^*,\gamma)$ is a contractible space, then the formula
$$
 \lambda(x,y,t)=\left\{\begin{array}{ll}
                         \gamma(x,2t), & 0\le t\le\frac{1}{2}, \\
                         \gamma(y,-2t+2), & \frac{1}{2}<t\le 1.
                       \end{array}
 \right.
 $$
 defines {a} continuous function $\lambda:X\times X\times [0,1]\to X$ with the required properties. Conversely, if $X$ is equiconnected, then fixing a point $x^*\in X$ and setting $\gamma(x,t)=\lambda(x,x^*,t)$, we obtain {that the space $(X,x^*,\gamma)$ is contractible}.

\begin{lemma}\label{l:4}
  Let  $0\le\alpha<\omega_1$, $X$ a topological space, $Y$  a contractible space, let $A_1,\dots,A_n$ be disjoint sets of the $\alpha$'th functionally multiplicative class in $X$ and \mbox{$f_i\in B_\alpha(X,Y)$}
for each $1\le i\le n$. Then there exists {a function $f\in B_\alpha(X,Y)$ such} that $f|_{A_i}=f_i$ for
each $1\le i\le n$. \end{lemma}

\begin{proof} Let $n=2$. {In view} of Lemma \ref{l:3} there exist functions $h_i\in B_\alpha(X,[0,1])$ such that $A_i=h_i^{-1}(0)$ for $i=1,2$. We set
$\displaystyle h(x)=\frac{h_1(x)}{h_1(x)+h_2(x)}$ for all $x\in X$. It is easy to verify that $h\in B_\alpha(X,[0,1])$ and  $A_i=h^{-1}(i-1)$, $i=1,2$.

Consider {a continuous function $\lambda:Y\times Y\times [0,1]\to Y$ such} that $\lambda(y,z,0)=y$ and  $\lambda(y,z,1)=z$ for all $y,z\in
Y$. Let $$f(x)=\lambda(f_1(x),f_2(x),h(x))$$ for every $x\in X$. Clearly, $f\in B_{\alpha}(X,Y)$. If  $x\in A_1$, then
$f(x)=\lambda(f_1(x),f_2(x),0)=f_1(x)$. If  $x\in A_2$, then $f(x)=\lambda(f_1(x),f_2(x),1)=f_2(x)$.

Assume that the lemma is true for all $2\le k<n$ and let $k=n$. According to our assumption, there exists a function $g\in
B_{\alpha}(X,Y)$ such that \mbox{$g|_{A_i}=f_i$} for all $1\le i<n$. Since $A=\bigcup\limits_{i=1}^{n-1}A_i$ and $A_n$ are disjoint sets which belong to the $\alpha$'th functionally multiplicative { class} in $X$, by the assumption, there is {a function
 $f\in B_\alpha(X,Y)$ with} $f|_A=g$
and $f|_{F_n}=f_n$. Then $f|_{F_i}=f_i$ for every $1\le i\le n$.
\end{proof}

Let $0\le \alpha<\omega_1$. We say that a mapping $f:X\to Y$ is {\it of the (functional) $\alpha$-th Lebesgue class},  $f\in H_\alpha(X,Y)$ ($f\in H_\alpha^*(X,Y)$), if the preimage $f^{-1}(V)$ belongs to the $\alpha$'th  (functionally) additive {class} in $X$ for any open set $V\subseteq Y$.

Clearly, $H_\alpha(X,Y)=H_\alpha^*(X,Y)$ for any perfectly normal space $X$.

The following statement is well-known, but we present a proof here {for convenience} of the reader.

\begin{lemma}\label{l:2}
 Let $X$ and $Y$ be topological spaces, {let} $(f_k)_{k=1}^\infty$ be a sequence of functions  \mbox{$f_k:X\to Y$} which is pointwise convergent to a function  $f:X\to Y$, {let} $F\subseteq Y$ be a closed set such that  $F=\bigcap\limits_{n=1}^\infty \overline{V}_n$, where $(V_n)_{n=1}^\infty$ is a sequence of open sets in $Y$ such that $\overline{V_{n+1}}\subseteq V_n$ for all $n\in\mathbb N$. {Then}
\begin{equation}\label{eq:1} f^{-1}(F)=\bigcap\limits_{n=1}^\infty\bigcup\limits_{k=n}^\infty f_k^{-1}(V_n). \end{equation}
 \end{lemma}

\begin{proof}
  Let $x\in f^{-1}(F)$ and $n\in\mathbb N$. Taking into account that $V_n$ is an open neighborhood of $f(x)$ and \mbox{$\lim\limits_{k\to\infty}f_k(x)=f(x)$}, we obtain that there is $k\ge n$ such that $f_k(x)\in V_n$.

   Now let $x$ {belong} to the {right-hand} side of~(\ref{eq:1}), i.e. for every  $n\in\mathbb N$ there exists a number $k\ge n$ such that $f_k(x)\in V_n$. Suppose
   $f(x)\not\in F$. Then there exists  $n\in\mathbb N$ such that  $f(x)\not\in \overline{V_{n}}$. Since $U=X\setminus  \overline{V_{n}}$ is a neighborhood of $f(x)$, there exists $k_0$ such that  $f_{k}(x)\in U$  for all $k\ge k_0$. In particular, $f_k(x)\in U$ for $k=\max\{k_0,n\}$. But then $f_k(x)\not\in V_n$, a contradiction. Hence, $x\in f^{-1}(F)$.
\end{proof}

\begin{lemma}\label{l:45}
 Let $X$ be a topological space, $Y$  a perfectly normal space and $0\le \alpha<\omega_1$. Then $B_\alpha(X,Y)\subseteq H^*_\alpha(X,Y)$ if $\alpha$
  is finite, and $B_\alpha(X,Y)\subseteq H^*_{\alpha+1}(X,Y)$ if $\alpha$ is infinite.
\end{lemma}

\begin{proof}
   Let $f\in B_\alpha(X,Y)$. Fix an arbitrary closed set $F\subseteq Y$. Since $Y$ is perfectly normal, there exists  a sequence of open sets $V_n\subseteq Y$ such that $\overline{V_{n+1}}\subseteq V_n$ and $F=\bigcap\limits_{n=1}^\infty \overline{V}_n$. {Moreover}, there exists a sequence of functions $f_k:X\to Y$ of Baire classes $<\alpha$ which is pointwise convergent to $f$ on $X$. By Lemma~\ref{l:2}, {equality} (\ref{eq:1}) holds. Denote  $A_n=\bigcup\limits_{k=n}^\infty f_k^{-1}(V_n)$.

If $\alpha=0$, then $f$ is continuous and  $f^{-1}(F)$ is a zero set in
$X$, since $F$ is a zero set in $Y$.

Suppose the assertion of the lemma is true for all finite ordinals $1\le\xi<\alpha$. We show that it is true for  $\alpha$. Remark that  $f_k\in
   B_{\alpha-1}(X,Y)$ for every $k\ge 1$. By the assumption, $f_k\in H^*_{\alpha-1}(X,Y)$ for every $k\in\mathbb N$. Then $A_n$ is of the functionally additive class  $\alpha-1$. Therefore, $f^{-1}(F)$ belongs to the $\alpha$'th functionally multiplicative  class.

  Assume the assertion of the lemma is true for all ordinals  $\omega_0\le\xi<\alpha$. For all $k\in\mathbb N$ we choose $\alpha_k<\alpha$ such that  $f_k\in   B_{\alpha_k}(X,Y)$ for every $k\ge 1$. The preimage $f_k^{-1}(V_n)$, being of the ($\alpha_k+1$)'th functionally additive {class}, belongs to the $\alpha$'th functionally additive class for all $k,n\in\mathbb N$, provided $\alpha_k+1\le\alpha$. Then $A_n$ is of the $\alpha$'th functionally additive class, hence, $f^{-1}(F)$ belongs to {the} $(\alpha+1)$'th functionally multiplicative  class.
\end{proof}

Recall that a family ${\mathcal A}=(A_i:i\in I)$ of sets $A_i$
{\it refines} a family \mbox{${\mathcal B}=(B_j:j\in J)$} of sets
$B_j$ if for every $i\in I$ there exists $j\in J$ such that
$A_i\subseteq B_j$. We write {in this case} ${\mathcal A}\preceq {\mathcal
B}$.

\section{PP-spaces and their properties}\label{sec:PP}

\begin{definition}\label{def:2} {\rm
  A topological space $X$ is said to be a {\it (strong) PP-space} if (for every dense set $D$ in $X$) there exist a sequence $((\varphi_{i,n}:i\in
  I_n))_{n=1}^\infty$ of locally finite partitions of unity on  $X$ and a sequence $((x_{i,n}:i\in I_n))_{n=1}^\infty$ of families of points of  $X$ (of $D$) such that
      \begin{equation}\label{eq:3}
       (\forall x\in X) ((\forall n\in\mathbb N \,\,\,x\in {\rm supp}\varphi_{i_n,n} ) \Longrightarrow (x_{i_n,n}\to x))
    \end{equation}}
\end{definition}

Remark that Definition~\ref{def:2} is equivalent to the following {one}.

\begin{definition}\label{def:1} {\rm
  A topological space $X$ is a {\it (strong) PP-space} if (for every dense set $D$ in $X$) there exist a sequence $((U_{i,n}:i\in
  I_n))_{n=1}^\infty$ of locally finite covers of $X$ by co-zero sets $U_{i,n}$ and a sequence $((x_{i,n}:i\in I_n))_{n=1}^\infty$ of families of points of  $X$ (of $D$) such that
       \begin{equation}\label{eq:2}
       (\forall x\in X) ((\forall n\in\mathbb N \,\,\,x\in U_{i_n,n} ) \Longrightarrow (x_{i_n,n}\to x))
  \end{equation}}
\end{definition}

Clearly, every strong PP-space is a PP-space.

\begin{proposition}
 Every metrizable space is {a} strong PP-space.
\end{proposition}

\begin{proof}
  Let $X$ be a metrizable space and $d$ a metric on $X$ which generates its topology. Fix an arbitrary dense set $D$ in $X$. For every $n\in\mathbb N$ let ${\mathcal B}_n$ be a cover of $X$ by open balls {of diameter} $\frac 1n$. Since $X$ is paracompact, for every $n$ there exists a locally finite cover \mbox{${\mathcal U}_n=(U_{i,n}:i\in I_n)$} of $X$ by open sets $U_{i,n}$ such that ${\mathcal U_n}\preceq {\mathcal B}_n$. Notice that each $U_{i,n}$ is a co-zero set. Choose a point $x_{i,n}\in D\cap U_{i,n}$ for all $n\in\mathbb N$ and $i\in I_n$. Let $x\in X$ and let $U$ be an arbitrary neighborhood of $x$. Then there is $n_0\in\mathbb N$ such that $B(x,\frac 1n)\subseteq U$ for all $n\ge n_0$. Fix $n\ge n_0$ and take $i\in I_n$ such that $x\in U_{i,n}$. Since ${\rm diam} \, U_{i,n}\le\frac 1n$, $d(x,x_{i,n})\le \frac 1n$, consequently, $x_{i,n}\in U$.
\end{proof}

\begin{example}
The Sorgenfrey line $\mathbb L$ is a strong PP-space which is not metrizable.
\end{example}

\begin{proof} Recall that the Sorgenfrey line is the real line $\mathbb R$ endowed with the topology generated by the base consisting of all semi-intervals $[a,b)$, where $a<b$ (see \cite[Example 1.2.2]{Eng}).

Let $D\subseteq \mathbb L$ be a dense set. For any $n\in\mathbb N$ and $i\in \mathbb Z$ {by $\varphi_{i,n}$ we denote} the characteristic function of $[\frac{i-1}{n},\frac{i}{n})$ and choose a point $x_{i,n}\in [\frac{i}{n},\frac{i+1}{n})\cap D$. Then the sequences  $\bigl((\varphi_{i,n}:i\in I_n) \bigr)_{n=1}^\infty$ and $\bigl((x_{i,n}:i\in I_n) \bigr)_{n=1}^\infty$ satisfy (\ref{eq:3}).
\end{proof}

\begin{proposition}\label{pr:3}
  Every $\sigma$-metrizable paracompact space is a PP-space.
\end{proposition}

\begin{proof}
  Let $X=\bigcup\limits_{n=1}^\infty X_n$, where $(X_n)_{n=1}^\infty$ is an increasing sequence of closed metrizable subspaces, and let $d_1$ be a metric on $X_1$ which generates its topology. According to Hausdorff's theorem \cite[p.~297]{Eng} we can extend the metric $d_1$ to a metric $d_2$ on $X_2$. Further, we extend the metric $d_2$ to a metric $d_3$ on $X_3$. Repeating this process, we obtain {a} sequence $(d_n)_{n=1}^\infty$ of metrics $d_n$ on $X_n$ such that $d_{n+1}|_{X_n}=d_n$ for every $n\in\mathbb N$. We define {a} function $d:X^2\to\mathbb R$ {by setting} $d(x,y)=d_n(x,y)$ {for} $(x,y)\in X_n^2$.

  Fix $n\in\mathbb N$ and $m\ge n$. Let ${\mathcal B}_{n,m}$ be a cover of $X_m$ by $d$-open balls of diameter $\frac 1 n$. For every $B\in {\mathcal B}_{n,m}$ there exists an open set $V_B$ in $X$ such that $V_B\cap X_m=B$. Let ${\mathcal V}_{n,m}=\{V_B:B\in {\mathcal B}_{n,m}\}$ and ${\mathcal U}_n=\bigcup\limits_{m=1}^\infty {\mathcal V}_{n,m}$. Then $\mathcal U_n$ is an open cover of $X$ for every $n\in\mathbb N$. Since $X$ is paracompact, for every $n\in\mathbb N$ there exists a locally finite partition of unity $(h_{i,n}:i\in I_n)$ on $X$ subordinated to $\mathcal U_n$. For every $n\in\mathbb N$ and $i\in I_n$ we choose $x_{i,n}\in X_{k(i,n)}\cap {\rm supp}\, h_{i,n}$, where $k(i,n)=\min\{m\in\mathbb N:X_m\cap {\rm supp}\, h_{i,n}\ne\O\}$.

  Now fix $x\in X$. Let $(i_n)_{n=1}$ be a sequence of indexes $i_n\in I_n$ such that $x\in {\rm supp}\, h_{i_n,n}$.  We choose $m\in\mathbb N$ such that $x\in X_m$. It is easy to see that $k(i_n,n)\le m$ for every $n\in\mathbb N$. Then $x_{i_n,n}\in X_m$. Since $d_m(x_{i_n,n},x)\le {\rm diam}\,{\rm supp\,}h_{i_n,n}\le\frac 1n$, $x_{i_n,n}\to x$ in $X_m$. Therefore, $x_{i_n,n}\to x$ in~$X$.
\end{proof}

Denote by $\mathbb R^\infty$ the collection of all sequences with a finite support, i.e. sequences of the form $(\xi_1,\xi_2,\dots,\xi_n,0,0,\dots)$, where $\xi_1,\xi_2,\dots,\xi_n\in\mathbb R$. Clearly, $\mathbb R^\infty$ is a linear subspace of the space $\mathbb R^{\mathbb N}$ of all sequences. Denote by $E$ the set of all sequences $e=(\varepsilon_n)_{n=1}^\infty$ of positive reals $\varepsilon_n$ and let $$U_e=\{x=(\xi_n)_{n=1}^\infty\in\mathbb R^\infty: (\forall n\in\mathbb N)(|\xi_n|\le \varepsilon_n)\}.$$
We consider on $\mathbb R^\infty$ the topology in which the system $\mathcal U_0=\{U_e:e\in E\}$ forms the base of neighborhoods of zero. Then $\mathbb R^\infty$ equipped with this topology is {a} locally convex $\sigma$-metrizable paracompact space which is not a first countable space, consequently, non-metrizable.

\begin{example}
  The space ${\mathbb R}^\infty$ is a PP-space which is not {a} strong PP-space.
\end{example}

\begin{proof}
 Remark that $\mathbb R^\infty$ is a PP-space by Proposition~\ref{pr:3}.

 We show that $\mathbb R^\infty$ is not a strong PP-space. Indeed, let
 $$
 A_n=\{(\xi_1,\xi_2,\dots,\xi_n,0,0,\dots):|\xi_k|\le \frac 1n\,\,(\forall 1\le k\le n)\},
 $$
 $$
 D=\bigcup\limits_{m=1}^\infty\bigcap\limits_{n=1}^m ({\mathbb R}^\infty\setminus A_n).
 $$
 Then $D$ is dense in $\mathbb R^\infty$, but there is no sequence in $D$ which converges to $x=(0,0,0,\dots)\in\mathbb R^\infty$. Hence, $\mathbb R^\infty$ is not a strong PP-space.
\end{proof}

\section{The Lebesgue classification}\label{sec:Leb}

The following result is an analog of theorems of K.~Kuratowski~\cite{Ku2} and D.~Montgomery~\cite{Mont} who proved that every separately continuous function, defined on a product of a metrizable space and a topological space and with values in a metrizable space, belongs to the first Baire class.

\begin{theorem}\label{th:2} Let $X$ be a strong PP-space, $Y$ a topological space, $Z$ a perfectly normal space and $0\le\alpha<\omega_1$. Then $$
C\overline{H_{\alpha}^*}(X\times Y,Z)\subseteq H_{\alpha+1}^*(X\times Y,Z). $$ \end{theorem}

\begin{proof}
  Let $f\in C\overline{H_{\alpha}^*}(X\times Y,Z)$. Then for the set $X_{H_\alpha^*}(f)$ there exist a sequence $(\mathcal U_n)_{n=1}^\infty$ of  locally finite covers $\mathcal U_n=(U_{i,n}:i\in I_n)$ of $X$ by co-zero sets $U_{i,n}$ and a sequence $((x_{i,n}:i\in I_n))_{n=1}^\infty$ of families of points of the set $X_{H_\alpha^*}(f)$ {satisfying} condition~(\ref{eq:2}).

  We choose an arbitrary closed set $F\subseteq Z$. Since $Z$ is perfectly normal, $F=\bigcap\limits_{m=1}^\infty
  {G_m}$, where $G_m$ are open sets in $Z$ such that $\overline{G}_{m+1}\subseteq G_m$ for every $m\in\mathbb N$. Let us verify that the equality
    \begin{equation}\label{eq:4}
     f^{-1}(F)=\bigcap\limits_{m=1}^\infty \bigcup\limits_{n\ge m}^\infty\bigcup\limits_{i\in I_n}U_{i,n}\times (f^{x_{i,n}})^{-1}(G_m).
    \end{equation}
  holds.
  Indeed, let $(x_0,y_0)\in f^{-1}(F)$. Then $f(x_0,y_0)\in G_m$ for every $m\in\mathbb N$. Fix any $m\in\mathbb N$. Since $V_m=f_{y_0}^{-1}(G_m)$ is {an} open neighborhood of $x_0$, there exists a number $n_0\ge m$ such that for all $n\ge n_0$ and $i\in I_n$ the inclusion $x_{i,n}\in V_m$ holds whenever $x_0\in U_{i,n}$. We choose $i_0\in I_{n_0}$ such that $x_0\in U_{i_0,n_0}$. Then $f(x_{i_0,n_0},y_0)\in G_m$. Hence, $(x_0,y_0)$ belongs to the {right-hand} side of (\ref{eq:4}).

  Conversely, let $(x_0,y_0)$ {belong} to the {right-hand} side of (\ref{eq:4}). Fix $m\in\mathbb N$. We choose sequences $(n_k)_{k=1}^\infty$, $(m_k)_{k=1}^\infty$ of numbers $n_k,m_k\in\mathbb N$ and a sequence $(i_k)_{k=1}^\infty$ of indexes $i_k\in I_{n_k}$ such that
  $$
  m=m_1\le n_1<m_2\le n_2<\dots<m_k\le n_k<\dots,
  $$
  $$x_0\in U_{i_k,n_k}\quad \mbox{and}\quad f(x_{i_k,n_k},y_0)\in G_{m_k}\subseteq G_m\quad \mbox{for every}\,k\in\mathbb N.$$
Since $\lim\limits_{k\to\infty}x_{i_k,n_k}=x_0$ and the function $f$ is continuous with respect to the first variable, $\lim\limits_{k\to\infty}f(x_{i_k,n_k},y_0)=f(x_0,y_0)$. Therefore, $f(x_0,y_0)\in\overline{G}_m$ for every $m\in\mathbb N$. Hence, $(x_0,y_0)$ belongs to the {left-hand} side of (\ref{eq:4}).

Since $f^{x_{i,n}}\in H_{\alpha}^*(Y,Z)$, the sets $(f^{x_{i,n}})^{-1}(G_m)$ are of the functionally additive class  $\alpha$ in $Y$. Moreover, all $U_{i,n}$ are co-zero sets in $X$, consequently, by \cite[Theorem~1.5]{K4} the set $E_n=\bigcup\limits_{i\in I_n}U_{i,n}\times
(f^{x_{i,n}})^{-1}(G_m)$ belongs to the $\alpha$'th  functionally additive {class} for every $n$. Therefore, $\bigcup\limits_{n\ge m} E_n$ is of the $\alpha$'th functionally additive {class}. Hence, $f^{-1}(F)$ is of the $(\alpha+1)$'th functionally multiplicative {class} in $X\times Y$.\hfill$\Box$
 \end{proof}

\begin{definition}{\rm
  We say that a topological space $X$ {\it  has {the} (strong) L-property} or {\it is {a} (strong) L-space}, if for every topological space  $Y$ every (nearly vertically) separately continuous function $f:X\times Y\to\mathbb R$ is of the first Lebesgue class.}
\end{definition}

According to Theorem~\ref{th:2} every strong PP-space has {the} strong L-property.

\begin{proposition}
  Let $X$ be a completely regular strong L-space. Then for any dense set $A\subseteq X$ and a point $x_0\in X$ there exists a countable dense set  $A_0\subseteq A$ such that $x_0\in\overline{A_0}$.
\end{proposition}

\begin{proof}
Fix an arbitrary everywhere dense set $A\subseteq X$ and a point $x_0\in A$. Let $Y$ be {the} space of all real-valued continuous functions on  $X$, endowed with {the} topology of pointwise convergence on $A$. Since the evaluation function \mbox{$e:X\times Y\to \mathbb R$}, $e(x,y)=y(x)$, is nearly vertically separately continuous, $e\in H_1(X\times Y,{\mathbb R})$. Then $B=e^{-1}(0)$ is $G_\delta$-set in $X\times Y$. Hence, $B_0=\{y\in Y: y(x_0)=0\}$ is {a} $G_\delta$-set in $Y$. We set $y_0\equiv 0$ and choose a sequence $(V_n)_{n=1}^\infty$ of basic neighborhoods of $y_0$ in $Y$ such that $\bigcap\limits_{n=1}^\infty V_n\subseteq B_0$. For every  $n$ there exist a finite set $\{x_{i,n}:i\in I_n\}$ of $X$ and $\varepsilon_n>0$ such that $V_n=\{y\in Y:\max\limits_{i\in I_n}|y(x_{i,n})|<\varepsilon_n\}$. Let
$$
A_0=\bigcup\limits_{n\in{\mathbb N}}\bigcup\limits_{i\in I_n}\{x_{i,n}\}.
$$
Take an open neighborhood $U$ of $x_0$ in $X$ and suppose that $U\cap A_0=\O$. Since $X$ is completely regular and $x_0\not\in X\setminus U$, there exists a continuous function $y:X\to \mathbb R$ such that $y(x_0)=1$ and $y(X\setminus U)\subseteq\{0\}$. Then $y\in \bigcap\limits_{n=1}^\infty V_n$, but $y\not\in B_0$, a contradiction. Therefore, $U\cap A_0\ne\O$, and $x_0\in \overline{A_0}$.
\end{proof}

\section{Baire classification and $\sigma$-metrizable spaces}\label{sec:Baire}

We recall that a topological space $Y$ is {\it $B$-favorable for a space $X$}, if $H_1(X,Y)\subseteq B_1(X,Y)$ (see \cite{KM11}).

 \begin{definition}{\rm
  Let $0\le\alpha<\omega_1$. A topological space $Y$ is called {\it weakly $B_{\alpha}$-favorable for a space $X$}, if
 $H_\alpha^*(X,Y)\subseteq B_\alpha(X,Y)$.}
 \end{definition}

Clearly, every $B$-favorable space is weakly $B_1$-favorable.

\begin{proposition}\label{pr:2} Let $0\le\alpha<\omega_1$, {let} $X$ be a topological space, $Y=\bigcup\limits_{n=1}^\infty Y_n$ a contractible space,
let $f:X\to Y$ be a mapping, $(X_n)_{n=1}^\infty$ a sequence of sets of the $\alpha$'th functionally additive {class} such that  $X=\bigcup\limits_{n=1}^\infty X_n$ and $f(X_n)\subseteq Y_n$ for every $n\in\mathbb N$. If one of the following conditions holds

\quad (i)  $Y_n$ is {a} nonempty weakly $B_\alpha$-favorable space for $X$ for all $n$ and $f\in H_\alpha^*(X,Y)$, or

\quad (ii) $\alpha>0$ and for every $n$ there exists a function $f_n\in B_{\alpha}(X,Y_n)$ such that  $f_n|_{X_n}=f|_{X_n}$,

\noindent then $f\in B_\alpha(X,Y)$.
\end{proposition}

\begin{proof} If $\alpha=0$ then the statement is obvious {in case} (i).

Let $\alpha>0$. By \cite[Lemma 2.1]{K4} there {exists} a sequence $(E_n)_{n=1}^\infty$ of disjoint functionally ambiguous sets of {the $\alpha$'th class} such that $E_n\subseteq X_n$ and $X=\bigcup\limits_{n=1}^\infty E_n$.

{In case} (i) for every $n$ we choose a point
$y_n\in Y_n$ and let
$$
f_n(x)=\left\{\begin{array}{ll}
                f(x), & \mbox{if}\,\,\, x\in E_n, \\
                y_n, & \mbox{if} \,\,\, x\in X\setminus E_n.
              \end{array}
\right.
$$
Since \mbox{$f\in H_\alpha^*(X,Y)$} and $E_n$ is functionally ambiguous set of the {$\alpha$'th class}, \mbox{$f_n\in H_\alpha^*(X,Y_n)$}. Then $f_n\in B_\alpha(X,Y_n)$ provided $Y_n$ is weakly $B_\alpha$-favorable for~$X$.

For every $n$ there exists a sequence of mappings
$g_{n,m}:X\to Y_n$ of classes $<\alpha$ such that $g_{n,m}(x)\mathop{\to}\limits_{m\to\infty} f_n(x)$ for every $x\in X$. In particular, $\lim\limits_{m\to\infty}
g_{n,m}(x)=f(x)$ on $E_n$. Since $E_n$ is of the $\alpha$-th functionally additive {class}, $E_n=\bigcup\limits_{m=1}^\infty B_{n,m}$, where
$(B_{n,m})_{m=1}^\infty$ is an increasing sequence of sets of functionally additive classes  $<\alpha$. Let $F_{n,m}=\O$ if $n>m$, and let
$F_{n,m}=B_{n,m}$ if $n\le m$. According to Lemma~\ref{l:4}, for every $m\in\mathbb N$ there exists a mapping $g_m:X\to Y$ of a class $<\alpha$ such that $g_m|_{F_{n,m}}=g_{n,m}$, since the system $\{F_{n,m}:n\in {\mathbb N}\}$ is finite for every  $m\in {\mathbb N}$.

It remains to prove that $g_m(x)\to f(x)$ on $X$. Let $x\in X$. We choose a number $n\in\mathbb N$ such that $x\in E_n$. Since the sequence
$(F_{n,m})_{m=1}^\infty$ is increasing, there exists a number $m_0$ such that $x\in F_{n,m}$ for all $m\ge m_0$. Then $g_m(x)=g_{n,m}(x)$ for all $m\ge m_0$. Hence, $\lim\limits_{m\to\infty} g_m(x)=\lim\limits_{m\to\infty}g_{n,m}(x)=f(x)$. Therefore, $f\in B_\alpha(X,Y)$.\hfill$\Box$
\end{proof}

\begin{definition}{\rm
  Let $\{X_n:n\in\mathbb N\}$ be a cover of a topological space $X$. We say that $(X,(X_n)_{n=1}^\infty)$ {\it has the property} $(*)$ if for every convergent sequence $(x_k)_{k=1}^\infty$ in $X$ there exists a number $n$ such that
  $\{x_k:k\in{\mathbb N}\}\subseteq X_n$.}
\end{definition}

\begin{proposition}\label{pr:4}
  Let $0\le\alpha<\omega_1$, {let} $X$ be a strong PP-space, $Y$ a topological space, {let} $(Z,(Z_n)_{n=1}^\infty)$ {have} the property $(*)$, {let} $Z_n$ {be} closed in $Z$ (and let $Z_n$ be a zero-set in $Z$ if $\alpha=0$) for every $n\in\mathbb N$, and \mbox{$f\in C\overline{B}_\alpha(X\times Y,Z)$}. Then there exists a sequence $(B_n)_{n=1}^\infty$ of sets of the $\alpha$'th /$(\alpha+1)$'th/ functionally multiplicative {class}  in $X\times Y$,  if $\alpha$ is finite /infinite/, such that $$\bigcup\limits_{n=1}^\infty B_n=X\times Y \qquad\mbox{and}\qquad f(B_n)\subseteq Z_n$$ for every  $n\in{\mathbb N}$. \end{proposition}

\begin{proof} Since  $X_{B_\alpha}(f)$ is dense in $X$, there exists a sequence $({\mathcal U}_m=(U_{i,m}:i\in I_m))_{m=1}^\infty$ of locally finite  co-zero covers of $X$ and a sequence  $((x_{i,m}:i\in I_m))_{m=1}^\infty$ of families of points of $X_{B_\alpha}(f)$ such {that condition}~(\ref{eq:2}) holds.

In accordance with \cite[Proposition~3.2]{M} there exists a pseudo-metric on $X$ such that all the set $U_{i,m}$ are co-zero with respect to this pseudo-metric. Denote by $\mathcal T$ the topology on $X$ generated by the pseudo-metric. Obviously, the topology $\mathcal T$ is weaker than the initial one. Using the paracompactness of $(X,{\mathcal T})$, for every $m$ we choose a locally finite open cover ${\mathcal V}_m=(V_{s,m}:s\in
  S_m)$ which refines ${\mathcal  U}_m$. By
\mbox{\cite[Lemma~1.5.6]{Eng}}, for every $m$ there exists a locally finite closed cover  \mbox{$(F_{s,m}:s\in S_m)$} of  $(X,{\mathcal
T})$ such that $F_{s,m}\subseteq V_{s,m}$ for every $s\in S_m$.  Now for every $s\in S_m$ we choose $i(s)\in I_m$ such that $F_{s,m}\subseteq U_{i(s),m}$.

For all $m,n\in{\mathbb N}$ and $s\in S_m$ let $$A_{s,m,n}=(f^{x_{i(s),m}})^{-1}(Z_n),\quad B_{m,n}=\bigcup\limits_{s\in S_m} (F_{s,m}\times
A_{s,m,n}),\quad B_n=\bigcap\limits_{m=1}^\infty B_{m,n}. $$ Since $f$ is of the $\alpha$'th Baire {class} with respect to the second variable, for every $n$ the set $A_{s,m,n}$ belongs to the $\alpha$'th functionally multiplicative {class} /$\alpha+1$/ in $Y$ for all $m\in\mathbb N$ and $s\in S_m$, if $\alpha$
is finite /infinite/ by Lemma~\ref{l:45}. According to~\cite[Proposition~1.4]{K4} the set $B_{m,n}$ is of the $\alpha$'th /$(\alpha+1)$'th/ functionally multiplicative { class}  in $(X,{\mathcal T})\times Y$. Then the set $B_n$ is of the $\alpha$'th /$(\alpha+1)$'th/ functionally multiplicative class in $(X,{\mathcal
T})\times Y$, and, consequently, in $X\times Y$ for every $n$.

We prove that $f(B_n)\subseteq Z_n$ for every $n$. To do this, fix  $n\in\mathbb N$ and $(x,y)\in B_n$. We choose a sequence  $(s_m)_{m=1}^\infty$ such that $x\in
F_{m,s_m}\subseteq U_{m,i(s_m)}$ and \mbox{$f(x_{m,i(s_m)},y)\in Z_n$}. Then $x_{m,i(s_m)}\mathop{\to}\limits_{m\to\infty} x$. Since $f$ is continuous with respect to the first variable, $f(x_{m,i(s_m)},y)\mathop{\to}\limits_{m\to\infty} f(x,y)$. The set $Z_n$ is closed, then $f(x,y)\in Z_n$.

Now we show that $\bigcup\limits_{n=1}^\infty B_n=X\times Y$. Let $(x,y)\in X\times Y$. Then there exists a sequence $(s_m)_{m=1}^\infty$ such that  $x\in
F_{m,s_m}\subseteq U_{m,i(s_m)}$ and $f(x_{m,i(s_m)},y)\mathop{\to}\limits_{m\to\infty} f(x,y)$. Since $(Z,(Z_n)_{n=1}^\infty)$ satisfies~$(*)$,
there is a number $n$ such that $\{f(x_{m,i_m},y):m\in {\mathbb N}\}$ is contained in $Z_n$, i.e. $y\in A_{m,n,i}$ for every $m\in\mathbb N$.
Hence, $(x,y)\in B_n$.\hfill$\Box$ \end{proof}

\begin{theorem}\label{th:3}
     Let $X$ be a strong PP-space, $Y$ a topological space, {$\{Z_n:n\in\mathbb N\}$ a} closed cover of a contractible perfectly normal space $Z$, {let $(Z,(Z_n)_{n=1}^\infty)$ satisfy}~$(*)$ and $Z_n$ {be} weakly $B_1$-favorable for
     $X\times Y$ for every $n\in\mathbb N$. Then
$$C\overline{C}(X\times Y,Z)\subseteq B_1(X\times Y,Z).$$
 \end{theorem}

\begin{proof}
  Let $f\in C\overline{C}(X\times Y,Z)$. In accordance with Theorem~\ref{th:2},  $f\in H_1^*(X\times Y,Z)$. Moreover, Proposition~\ref{pr:4} implies that there exists a sequence of zero-sets  $B_n\subseteq X\times Y$ such that
$\bigcup\limits_{n=1}^\infty B_n=X\times Y$ and $f(B_n)\subseteq Z_n$ for every $n\in{\mathbb N}$. Since for every $n$ the set $B_n$ is an $F_\sigma^*$-set and $H_1^*(X\times Y,Z_n)\subseteq B_1(X\times Y,Z_n)$, $f\in B_1(X\times Y,Z)$ by Proposition~\ref{pr:2}.\hfill$\Box$ \end{proof}

 \begin{definition}{\rm
   A topological space $X$ is called {\it strongly $\sigma$-metrizable}, if it is $\sigma$-metrizable with a stratification $(X_n)_{n=1}^\infty$ and $(X,(X_n)_{n=1}^\infty)$ has the property~$(*)$.}
 \end{definition}

Taking into account that every regular strongly $\sigma$-metrizable space with metrizable separable stratification is perfectly normal (see \cite[Corollary~4.1.6]{Masl_dis}) and every metrizable separable arcwise connected and  locally arcwise connected space is weakly $B_\alpha$-favorable for any topological space $X$ for all
$0\le\alpha<\omega_1$ \cite[Theorem~3.3.5]{Dis}, we immediately obtain the following corollary of Theorem~\ref{th:3}.

 \begin{corollary}
   Let $X$ be a strong PP-space, $Y$ a topological space and $Z$ a contractible regular strongly $\sigma$-metrizable space with a stratification $(Z_n)_{n=1}^\infty$, where $Z_n$ is a metrizable separable arcwise connected and  locally arcwise connected space for every $n\in\mathbb N$. Then $$C\overline{C}(X\times Y,Z)\subseteq B_1(X\times Y,Z).$$
 \end{corollary}

\end{document}